\documentclass[a4paper]{amsart}
\usepackage[english]{babel}

\usepackage{amssymb}
\usepackage{amsmath,amsthm}
\usepackage{mathtools}
\usepackage{dsfont}
\usepackage{nicefrac}
\usepackage{cases}
\usepackage{bm}
\usepackage{enumitem}

\usepackage[T1]{fontenc}
\usepackage{lmodern}
\usepackage{microtype}

\usepackage[a4paper, inner=3cm,outer=3cm,top=3.3cm,bottom=3cm]{geometry}

\usepackage{xspace}

\usepackage[autostyle]{csquotes}
\usepackage[dvipsnames]{xcolor}

\usepackage{tikz}
\usetikzlibrary{matrix,arrows,patterns,decorations.pathreplacing,svg.path}
\usepackage{pgfplots}

\usepackage{tikz-3dplot}
\usetikzlibrary{calc}
\tdplotsetmaincoords{60}{110}

\usepackage{todonotes}

\usepackage[colorlinks, bookmarksdepth=3, citecolor=blue,linkcolor=blue]{hyperref}

\usepackage{pdfcomment}

\makeatletter
\@ifpackageloaded{hyperref}{%
\DeclareRobustCommand{\SkipTocEntry}[5]{}}{%
\DeclareRobustCommand{\SkipTocEntry}[4]{}}
\makeatother

\theoremstyle{plain}
\newtheorem{Th}{Theorem}
\newtheorem{Prop}[Th]{Proposition}

\newtheorem{Lem}[Th]{Lemma}

\theoremstyle{definition}
\newtheorem{Def}[Th]{Definition}

\newtheorem{Rem}[Th]{Remark}

\DeclareMathOperator{\conv}{conv}

\newcommand{\cal}[1]{\mathcal{#1}}

\newcommand{\R}{\mathbb R}

\newcommand{\cP}{\mathcal P}

\newcommand{\cX}{\cal X}

\newcommand{\cQ}{\cal Q}

\DeclareMathOperator{\V}{V}

\DeclareMathOperator{\Vol}{Vol}

\DeclareMathOperator{\Gr}{Gr}

\DeclareMathOperator{\Abs}{Abs}

\newcommand{\re}[1]{(\ref{e:#1})}

\newcommand{\rf}[1]{Figure~\ref{F:#1}}

\makeatletter
\@namedef{subjclassname@2020}{%
  \textup{2020} Mathematics Subject Classification}
\makeatother

\title[Bezout-type inequality for mixed volumes of three-dimensional zonoids]{An algebraic-combinatorial proof of a Bezout-type inequality for mixed volumes of three-dimensional zonoids}
\author{Gennadiy~Averkov}
\address{Fakult\"at 1, BTU Cottbus-Senftenberg, Platz der Deutschen Einheit 1, 03046 Cottbus, Germany}
\email{averkov@b-tu.de}
\author{Ivan~Soprunov}
\address{Department of Mathematics and Statistics, Cleveland State University,  2121 Euclid Ave, Cleveland, Ohio, 44115 USA}
\email{i.soprunov@csuohio.edu}

\begin{document}
\selectlanguage{english}

\date{}


\maketitle

\begin{abstract} 
	We present a new algebraic-combinatorial approach to proving a Bezout-type inequality for zonoids in dimension three, which has recently been established by Fradelizi, Madiman, Meyer, and Zvavitch. Our approach hints at connections between inequalities for mixed volumes of zonoids and real algebra and matroid theory.  
\end{abstract}

\section{Introduction}

The Brunn-Minkowski theory emerges from the fundamental result of Minkowski, asserting that the $d$-dimensional volume 
of the Minkowski-linear combination $\lambda_1A_1+\dots+\lambda_kA_k$ of convex bodies $A_1,\ldots,A_k$ in $\R^d$ with
non-negative scalar multiples $\lambda_1,\ldots,\lambda_k \in \R_{\ge 0}$  is a homogeneous polynomial of degree $d$
in $\lambda_1,\dots,\lambda_k$:
\[
	f_{A_1,\ldots,A_k} (\lambda_1,\ldots,\lambda_k)=\Vol_d(\lambda_1A_1+\dots+\lambda_kA_k).
\] 
The polynomial $f_{A_1,\ldots,A_k}$ is called the {\it volume polynomial} of the system of bodies $A_1,\ldots,A_k$.  
The  coefficients of the volume polynomial are related to the so-called mixed volume functional. It is a symmetric real-valued functional $\V(K_1,\ldots,K_d)$  of $d$ convex bodies $K_1,\ldots,K_d$, which is uniquely defined by the property that it is Minkowski-linear in each of the $d$ bodies and coincides with the volume functional when the bodies are the same, see \cite[Section 5.1]{Schneider2014}.
Expressing $\Vol_d(\lambda_1A_1+\dots+\lambda_kA_k)$ as the mixed volume and using the multilinearity, one obtains
\[
	f_{A_1,\ldots,A_k} (\lambda_1,\ldots,\lambda_k) = \sum_{1\leq i_1,\ldots,i_d  \leq k} \lambda_{i_1} \cdots \lambda_{i_k} \V(A_{i_1},\ldots,A_{i_d}).
\]
This shows that the coefficients of the volume polynomial $f_{A_1,\ldots,A_k}$ are determined by the set of all mixed volumes 
$\V(A_{i_1},\ldots,A_{i_d})$ that can be produced by the $k$ bodies $A_1,\ldots,A_k$. For example, for three bodies $A_1,A_2,A_3$, one has 
\[
	f_{A_1,A_2,A_3}(\lambda_1,\lambda_2,\lambda_3) =  \sum_{i=1}^3 \V(A_i,A_i,A_i)\lambda_i^3+
	3\!\!\sum_{\substack{1\leq i,j\leq 3\\ i \ne j}}\!\!\V(A_i,A_i,A_j)\lambda_i^2\lambda_j+6 \V(A_1,A_2,A_3) \lambda_1 \lambda_2 \lambda_3.
\] 
Thus, three bodies in dimension three produce $10$ mixed volumes: $3$ volumes $\V(A_i,A_i,A_i)$,
$6$ mixed volumes of the form $\V(A_i,A_i,A_j)$ with $i \ne j$, and one mixed volume $\V(A_1,A_2,A_3)$. A large portion of the Brunn-Minkowski theory is concerned with finding relations between mixed volumes of a given system of convex bodies in terms of inequalities. The most famous such inequality is the Aleksandrov-Fenchel inequality  \cite[Section 7.3]{Schneider2014}
$$\V(A_1,A_1,A_3,\ldots,A_d) \V(A_2,A_2,A_3,\ldots,A_d) \le \V(A_1,A_2,A_3,\ldots,A_d)^2.$$ 
Mixed volumes and inequalities relating them are extremely useful in a wide range of applications, including stochastic geometry, statistics, algebra, algebraic geometry, and combinatorics. See, for example, \cite{Bernstein75, BH20, ZonoidAlgebra, Stan}, as well as \cite[Chapter 4]{BuZa}. One of the most challenging goals of the Brunn-Minkowski theory is to give an exhaustive description of relations between mixed volumes for a given number of bodies $k$ in a given dimension $d$. Equivalently, this can be phrased as the problem of characterizing all volume polynomials $f_{A_1,\ldots,A_k}$ of degree $d$ in $k$ variables. 
This problem goes back to the works of Heine \cite{Heine38} who solved the case of $k=2,3$ in dimension $d=2$
and Shephard \cite{Shephard60} who extended the case of $k=2$ to arbitrary dimension $d$.
All other cases of the Heine-Shephard problem remain open with the smallest unsolved cases being
$(k,d)=(2,4)$ and $(k,d)=(3,3)$. In \cite{Gur09} Gurvits conjectured that the space of all volume polynomials for $(k,d)=(3,3)$
equals the space of homogeneous strongly log-concave polynomials or, equivalently, Lorentzian polynomials \cite{BH20}.
This conjecture was disproved by Brändén and Huh \cite{BH20, Huh23}. See also \cite{menges2023} for most recent results on volume polynomials versus Lorentzian polynomials. In \cite{AS22} the authors gave a partial solution to the Heine-Shephard problem in the case of $(k,d)=(2,4)$, where they obtained a complete inequality description of the
space of square-free parts of volume polynomials.

Analogs of the Heine-Shephard problem can also be studied for more specific classes of convex bodies, which play a special role in convexity and  applications. 
There is an interesting class of centrally-symmetric convex bodies called zonoids.
Zonoids are Hausdorff limits of zonotopes, while zonotopes are defined as Minkowski sums of finitely many segments. 
Zonoids and their mixed volumes appear in a wide range of fields, including algebraic geometry, combinatorics, and statistics, see \cite{FZ23} and references therein.
Hence, within the general study of mixed volume, it is natural to narrow the scope to study relations between mixed volumes of zonoids. More new inequalities for mixed volumes of zonoids in dimension 2 and 3 were obtained in 
\cite{ADRS24} using polyhedral and computer assisted methods.

The aim of this note is to give a structured algebraic-combinatorial proof of a new inequality for mixed volumes of zonoids, which was recently obtained by Fradelizi, Madiman, Meyer, and Zvavitch, see Theorem 5.1 and equation (37) in 
 \cite{FZ23}. 
\begin{Th}\label{main:thm} 
	The mixed volumes of zonoids $A,B,C \subset \R^3$ satisfy 
	\begin{equation} \label{the:ineq} 
		\V(A,A,A) \V(A,B,C) \le \frac{3}{2} \V(A,A,B) \V(A,A,C). 
	\end{equation} 
\end{Th} 
There is a general inequality which follows from the Aleksandrov-Fenchel inequality
\begin{equation} \label{e:square} 
\V(A,A,A_3,\ldots,A_d) \V(B,C,A_3,\ldots,A_d) \le 2\V(A,B,A_3,\ldots,A_d)\V(A,C,A_3,\ldots,A_d),
\end{equation}
valid for arbitrary convex bodies $A,B,C,A_3,\dots, A_d$ in $\R^d$, see
\cite[Lemma 5.1]{BGL}, where the constant 2 is best possible. For example, in dimension three \re{square} becomes equality if one takes
$A$ to be the square pyramid $A = \conv( 0, e_1, e_2, e_1+e_2, e_3)$ and $B = [0,e_1]$ and $C= [0,e_2]$ to be unit segments, where $e_1,e_2,e_3$ are the standard basis vectors. 
Thus, the above theorem shows that, in the class of zonoids, inequality \re{square} is no longer sharp. 
As a consequence, this shows that the space of volume polynomials of 3-dimensional zonoids is properly contained in the space of all volume polynomials. 

As mentioned above, the Heine-Shephard problem is open for three bodies in $\R^3$, but may be more
manageable in the case of zonoids.
It would be interesting to know whether (\ref{the:ineq}) would occur in a complete inequality description 
of the space of volume polynomials of three zonoids in $\R^3$.

\section{The proof}

Recall that a subset of $\R^d$ is a \emph{convex body} if it is compact, convex, and non-empty. The Hausdorff metric of convex bodies $K$ and $L$ is the least $\rho \ge 0$ such that every point of $K$ is at at distance at most $\rho$ to some point of $L$ and, vice versa, every point of $L$ is at distance at most $\rho$ to some point of $K$. 
The Minkowski sum of $K,L \subseteq \R^d$ is defined by $K+ L = \{ p+ q \colon p \in K, \ q \in L\}$ and the non-negative scaling of $K \subseteq \R^d$ by $\lambda \in \R_+$ is defined by $\lambda K = \{ \lambda p \colon p \in K\}$. 
A \emph{zonotope} is a convex polytope which is the Minkowski sum of finitely many segments, and a \emph{zonoid} is defined as the limit of a sequence of zonotopes in the Hausdorff metric.

The theory of mixed volumes is developed for arbitrary dimension $d$, but here we restrict our attention to dimension $d=3$. 
Let $A,B,C$ be convex bodies in $\R^3$.  We list the defining properties of the mixed volume of convex bodies in dimension three: 
\begin{enumerate}[label= {\bfseries (P\arabic*)} ]
	\item $\V(A,B,C)$ is non-negative and does not depend on the order, in which $A,B,C$ occur in the mixed volume. 
	\item $\V(A,B,C)$ is Minkowski-linear in each of the three bodies $A,B,C$, which means that $\V( \sum_{i=1}^m \alpha_i A_i, B, C) = \sum_{i=1}^m \V(A_i,B,C)$ holds when $\alpha_i\ge 0 $ and $A_i,B,C$ are convex bodies in $\R^3$.  (Note that in view of (P1), the linearity also holds with respect to $B$ and $C$.)
	\item $\V(A,A,A)$ is the volume $\Vol_3(A)$ of $A$. 
\end{enumerate} 
We also rely on the following basic and well-known properties of mixed volumes (again, restricting the attention to the three-dimensional case): 
\begin{enumerate}[start = 4, label = {\bfseries  (P\arabic*)}]
	\item \label{lin:trans} $\V(\phi(A), \phi(B), \phi(C)) = | \det(\phi) | \V(A,B,C)$ for every linear transformation $\phi : \R^3 \to \R^3$.
	\item $\V(A,B,C)=0$ if two of the three bodies $A,B,C$ are parallel segments. 
	\item $\V(A',B', C') = \V(A,B,C)$ when $A', B', C'$ are translates of $A, B$ and $C$, respectively. 
	\item $\V(A,B,C)$ is continuous in the Hausdorff metric in each of the three bodies $A, B$ and $C$. 
	\item \label{mv:det} If $A, B,C$ are segments $A = [0,a], B = [0,b], C = [0,c]$ joining the origin with $a \in \R^3$, $b \in \R^3$ and $c \in \R^3$, respectively, then $\V(A,B,C) = \frac{1}{3!} | \det(a,b,c)|$. 
\end{enumerate} 
 As mentioned in \cite{FZ23}, it is not hard to see that for verifying \eqref{the:ineq} on zonoids, it is enough to verify it on zonotopes, and for the verification of \eqref{the:ineq} on zonotopes, it is enough to deal with the case $A = \sum_{i=1}^m [0,a_i]$, $B = [0,e_1]$ and $C = [0,e_2]$, where $a_i = (x_i,y_i,z_i) \in \R^3$, $e_1= (1,0,0), e_2 = (0,1,0)$.  Such a reduction can be easily justified by the properties of mixed volumes, listed above. As \eqref{the:ineq} is additive in both $B$ and $C$, which are zonotopes, using the additivity of the mixed volume, it is enough to consider the case when both $B$ and $C$ are segments. By the continuity of the mixed volumes, it is enough to handle the generic case of non-parallel segments $B$ and $C$. But then, in view of \ref{lin:trans}, one can fix $B = [0,e_1]$ and $C =[0,e_2]$. 

We may assume that $m \ge 3$, since for $m < 3$,  one has $\Vol_3(A) =0$ and the assertion is trivially satisfied. Having carried out the latter reduction, one then can use \ref{mv:det} to reformulate \eqref{the:ineq} as the inequality for the absolute values of certain minors of the $3 \times m$ matrix $(a_1,\ldots,a_m) = \begin{pmatrix} x_1 & \cdots & x_m \\ y_1 & \cdots & y_m \\ z_1 & \cdots & z_m \end{pmatrix}$. In this way, in \cite[Section 5]{FZ23} it was shown that Theorem~\ref{main:thm} is equivalent to the following lemma. 
\begin{Lem}\label{L:FMMZ}
Let $m\geq 3$. Then for any collection of $m$ vectors
$\{(x_i,y_i,z_i) : 1\leq i\leq m\}$ in $\R^3$ one has
\begin{align}\label{e:matrix}
\sum_{1 \le i < j < k \le m}  \left |\det\!\!\begin{pmatrix}
x_i & x_j & x_k \\
y_i & y_j & y_k \\
z_i & z_j & z_k
\end{pmatrix}\right|  \sum_{i=1}^m |z_i|  \le \!\!\!\sum_{1 \le i < j  \le m} \left |\det\!\!\begin{pmatrix}
y_i & y_j \\
z_i & z_j 
\end{pmatrix}\right|\sum_{1 \le i < j \le m} \left |\det\!\!\begin{pmatrix}
x_i & x_j  \\
z_i & z_j 
\end{pmatrix}\right|.
\end{align}
\end{Lem} 

Our contribution is a systematic proof of this lemma, which adheres to combinatorial and algebraic principles. The right-hand side of  \eqref{e:matrix} are functions of the three groups of variables:  $x = (x_1,\ldots,x_m)$, $y = (y_1,\ldots,y_m)$ and $ z = (z_1,\ldots,z_m)$. If $y$ and $z$ are fixed and $x$ is varied, then these functions are continuous piecewise linear in $x$. If $x$ and $z$ are fixed, they are continuous piecewise linear in $y$. Following the original proof, we heavily rely on this qualitative observation, but we use a different proof strategy. 
The original proof uses Bonnesen's inequality in one of the proof steps, whereas our proof is self-contained.  The process of deriving estimates in the original proof is rather intricate, while in our proof, after we exploit certain qualitative properties of the functions in \eqref{e:matrix}, we derive \eqref{e:matrix} in one shot. Also, in the original proof,
 one observes the behavior of  piecewise affine functions on $\R$ at $+\infty$ and $-\infty$, which is a technicality we can avoid in our proof. 
 
 There is yet another qualitative aspect of \eqref{e:matrix}, which is important. When $z$ is fixed, the left-hand side of \eqref{e:matrix} is a conic combination of the absolute values of determinants. Determinants are multi-linear in each of their rows so that the absolute values of the determinants are multi-convex in each of their rows. It follows that, when $y, z$ are fixed, the left-hand side of \eqref{e:matrix} is convex in $x$, and if $x, z$ are fixed, then it is convex in $y$. We introduce a general formalism which captures the crucial qualitative properties of \eqref{e:matrix}.

\begin{Def} Let $\cP=\{P_1, \ldots, P_M\}$ be a family of finitely many $m$-dimensional polyhedra with $P_1 \cup \cdots \cup P_M = \R^m$. We call such $\cP$ a \emph{polyhedral subdivision} of $\R^m$ if for any two distinct polyhedra $P_i$ and $P_j$ in this family the intersection $P_i \cap P_j$ is either empty or a proper face of both $P_i$ and $P_j$. 
We call $G \subseteq \R^m$ a \emph{generating set} of $\cP$ if each $P_i$ is the convex hull of some subset $G_i$ of $G$. 
\end{Def} 

For example, when $\cP$ is a family of pointed polyhedra, then the union of all vertices and unbounded one-dimensional faces of polyhedra in $\cP$ a generating set of $\cP$. See also an example in Fig.~\ref{ex:subdivis:gen:set}.

\begin{figure} 
\begin{tikzpicture}[scale = 0.8]
\draw[color=orange,line width = 2pt] (0,2) -- (3,2) -- (3,0);
\draw[color=orange,line width = 2pt] (0,3) -- (3,3); 
\draw[color=orange,line width = 2pt] (4,0) -- (4,2);
\draw[color=orange,line width = 2pt] (5,4) -- (5,6);
\draw[color=orange,line width = 2pt] (6,0) -- (6,3) -- (8,3);
\draw[line width = 1pt] (3,2) -- (4,2) -- (5,3) -- (6,3) -- (5,4) -- (5,3) ;
\draw[line width = 1pt] (3,2) -- (3,3) -- (4,3) -- (4,2);
\draw[line width = 1pt] (4,3) -- (5,4);
\foreach \x/\y in {3/2,4/2,3/3,4/3,5/3,6/3,5/4}
{
	\fill[orange] (\x,\y) circle (0.12);
}
\end{tikzpicture} 
\caption{A polyhedral subdivision of $\R^2$ and its generating set, being a union of seven rays and two points.}\label{ex:subdivis:gen:set}
\end{figure}
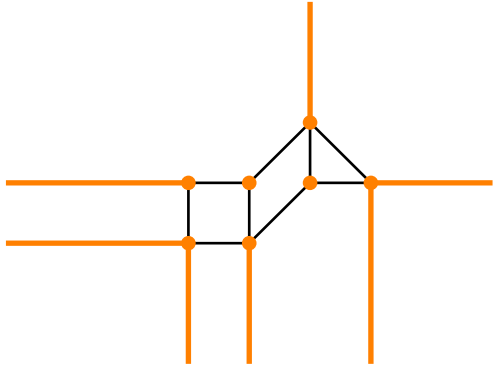 

\begin{Def} We say that $g: \R^m \to \R$ is a \emph{piecewise affine function} if there is a polyhedral subdivision $\cP = \{P_1,\ldots,P_M\}$ of $\R^m$ such that $g$ is affine on each $P_i$, for $1 \le i \le M$. In this case, we say that $\cP$ is consistent with $g$. 
\end{Def} 

\begin{Def} 
We say a function $f : \R^m \times \R^n \to \R$ is {\it bi-convex} if $f(x,y)$ is convex with respect to $x$ for each fixed $y \in \R^n$ and with respect to $y$ for each fixed $x \in \R^m$. 
\end{Def}


\begin{Prop}\label{P:abstract}
	Let  $\cP$ and  $\cQ$ be polyhedral subdivisions of $\R^m$ and $\R^n$ with generating sets $G$ and $H$, respectively.
	 Let $f : \R^m \times \R^n \to \R$ be a bi-convex function and $g : \R^m \to \R$ and $h : \R^n \to \R$ 
	piecewise affine functions, such that $\cP$ and  $\cQ$ are consistent with $g$ and $h$, respectively. Then the inequality 
	\(f(x,y) \le g(x) h(y)
	\) is valid for all $(x,y) \in \R^m\times\R^n$ if and only if it is valid for all $(x,y) \in G\times H$. 
 \end{Prop}

\begin{proof} The ``only if'' implication is trivial, so assume that the inequality $f(x,y) \le g(x) h(y)$ is valid for all $(x,y) \in G\times H$. Considering arbitrary $(x,y)\in \R^m\times\R^n$, we want to verify $f(x,y) \le g(x) h(y)$. Let $\cP = \{P_1,\ldots, P_M\}$ and $\cQ = \{Q_1,\ldots,Q_N\}$. Furthermore, let $P_i = \conv( G_i)$ and $Q_j = \conv(H_j)$, where $G_i \subseteq G$ and $H_j \subseteq H$. 
	
	Pick $i$ and $j$ such that $x \in P_i$ and $y \in G_j$ and write $x$ and $y$ as convex combinations $x = \sum_{s \in S} \lambda_s x_s$ and $y = \sum_{t \in T} \mu_t y_t$, where $S$ and $T$ are finite index sets, $x_s \in G_i$ and $y_t \in H_j$. We have 
	\begin{align*} 
		f(x,y) &  = f\bigl( \sum_{s \in S} \lambda_s x_s, \sum_{t \in T} \mu_t y_t\bigr) 
			\\ & \le \sum_{s \in S, t \in T} \lambda_s \mu_t f(x_s,y_t) & & \text{(by bi-convexity of $f$)}
			\\ & \le \sum_{s \in S, t \in T} \lambda_s \mu_t g(x_s) h(y_t) & & \text{(by assumption, in view of $x_s \in G_i, y_t \in H_j$)}
			\\ & = \bigl( \sum_{s \in S} \lambda_s g(x_s) \bigr) \bigl( \sum_{t \in T} \mu_t h(y_t) \bigr) 
			\\ & = g \bigl( \sum_{s \in S} \lambda_s x_s \bigr) h \bigl( \sum_{t \in T} \mu_t y_t \bigr) & & \text{(by affine linearity of $g$ on $P_i$ and $h$ on $Q_j$)}. 
			\\ & = g(x) h(y). 
	\end{align*}
\end{proof}

\begin{proof}[Proof of Lemma~\ref{L:FMMZ}] 
Since the left- and the right-hand side of \eqref{e:matrix} are continuous in $z$, it suffices to verify the assertion in the generic case $z_i \ne 0$ for each $i=1,\ldots,m$. The degenerate case, with some of the $z_i$'s equal to $0$, follows by taking the limit of the generic case.  We assume that the values $z_i \ne 0$ are fixed and view the left- and the right-hand side of \eqref{e:matrix} as functions of $x$ and $y$. The inequality we want to verify has  the structure $f(x,y) \le g(x) h(y)$ as described in Proposition~\ref{P:abstract}. In our case, we have $g=h$, so that the inequality that we prove is $f(x,y) \le g(x) g(y)$. 

In order to use Proposition~\ref{P:abstract}, we first determine a polyhedral subdivision of $\R^m$ consistent with~$g$. The function $g(x)$ is the sum of functions $g_{ij} (x) := \left| \det \begin{pmatrix} x_i & x_j \\ z_i & z_j \end{pmatrix} \right|$ with $1 \le i < j \le m$.  The function $g_{ij}(x)$ is affine  on both sides of the hyperplane  $X_{ij} := \{ x \in \R^m \colon x_i / z_i = x_j / z_j \}$. The condition $x_i / z_i = x_j / z_j$ means that the vectors $(x_i, z_i)$ and $(x_j,z_j)$ have the same slope. It follows that the function $g(x) = \sum_{1 \le i < j \le m} g_{ij}(x)$ is affine on each region of the hyperplane arrangement $\cX = \{X_{ij} \colon 1 \le i < j \le m\}$. 

We note that $\cX$ is essentially the well-known \emph{braid arrangement}:\footnote{$\cX$ is the braid arrangement when all $z_1 = \cdots = z_m =1$, but the case $z_i \ne 0$ is equivalent to this specific case up to a non-singular linear transformation.} each of the $m!$ regions of $\cX$ corresponds to a way of sorting the values $x_i / z_i$. That is, each region is given by the system of inequalities $x_{\sigma(1)}/z_{\sigma(1)} \le \cdots \le x_{\sigma(m)} / z_{\sigma(m)}$, for some permutation $\sigma\in S_m$. See also Fig.~\ref{fig:braid} for an illustration for $m=3$. 
Let us denote such a region by $R_\sigma$. It is clear that the polyhedral subdivision 
$\cP = \{ R_\sigma : \sigma\in S_m \}$ is consistent with $g$. 
The region $R_\sigma$  is a polyhedral cone whose lineality space is
one-dimensional, given by the equations $x_1/ z_1 = \cdots = x_m / z_m$, and whose two-dimensional faces are given by the conditions $x_{\sigma(1)} / z_{\sigma(1)}  = \cdots = x_{\sigma(\ell)} / z_{\sigma(\ell)} \le x_{\sigma(\ell+1)} / z_{\sigma(\ell+1)}  = \cdots = x_{\sigma(m)} / z_{\sigma(m)}$, for $\ell =1,\ldots,m-1$. 
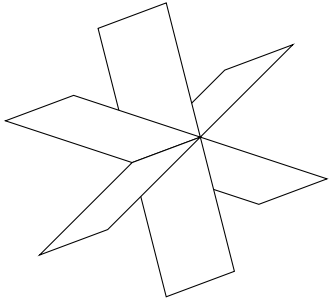
\begin{figure}[h]
\begin{center}
	\begin{tikzpicture}[scale=1.5, tdplot_main_coords]
		\filldraw[fill=white] (0,0,0) -- (1,1,1) -- (0,3/2,3/2) -- (-1,1/2,1/2) -- cycle; 
		\filldraw[fill=white] (0,0,0) -- (1,1,1) -- (1/2,1/2,2) -- (-1/2,-1/2,1) -- cycle; 
		\filldraw[fill=white] (0,0,0) -- (1,1,1) -- (1/2,2, 1/2) -- (-1/2,1,-1/2) -- cycle; 
		\filldraw[fill=white] (0,0,0) -- (1,1,1) -- (3/2,3/2,0) -- (1/2,1/2,-1) -- cycle; 
		\filldraw[fill=white] (0,0,0) -- (1,1,1) -- (3/2,0,3/2) -- (1/2,-1,1/2) -- cycle; 
		\filldraw[fill=white] (0,0,0) -- (1,1,1) -- (2,1/2, 1/2) -- (1,-1/2,-1/2) -- cycle; 		
	\end{tikzpicture}
\end{center}
\caption{The braid arrangement in dimension $3$. The respective three planes, given by the equations $x_1 = x_2$, $x_1 = x_3$ and $x_2 =x_3$, intersect in the line given by the equations $x_1 = x_2 =x_3$. The the planes decompose the space into six cells, with each cell determined by a permutation $\sigma \in  S_3$ and defined by the inequalities $x_{\sigma(1)} \le x_{\sigma(2)} \le x_{\sigma(3)}$.} \label{fig:braid}
\label{F:braid}
\end{figure}
Since the lineality space of $R_\sigma$ is one-dimensional, $R_\sigma$ is the convex hull of its two-dimensional faces, as illustrated in \rf{braid}.

It follows that the set $G$ consisting of all $x \in \R^m$, for which the sequence $x_1/z_1,\ldots,x_m/z_m$ contains at most two distinct values, is a generating set of $\cP$. Therefore, by Proposition~\ref{P:abstract} it is enough to verify $f(x,y) \le g(x) g(y)$ for all $x,y \in G$. 

Consider arbitrary $x, y \in G$. There exist partitions $\{E, E'\}$ and $\{F,F'\}$ of $\{1,\ldots,m\}$ and numbers $\lambda, \lambda, \mu, \mu' \in \R$ such that $x_i / z_i = \lambda$ for all $i \in E$, $x_i/ z_i = \lambda'$ for all $i \in E'$, $y_i / z_i = \mu$ for all $i \in F$ and $y_i / z_i = \mu'$ for all $i \in F'$. Under these conditions, we have 
\begin{align*} 
	g(x) &  =  \sum_{i \in E, i' \in E'} g_{ii'} (x) =  \sum_{i \in E, i' \in E'} |\lambda- \lambda'| |z_i| |z_{i'}|  = |\lambda - \lambda' | \left(\sum_{i \in E} |z_i| \right) \left(\sum_{i' \in E'} |z_{i'}| \right), \\
	g(y) &  =  \sum_{j \in F, j' \in F'} g_{jj'} (y) =  \sum_{j \in F, j' \in F'} |\mu- \mu'| |z_j| |z_{j'}|  = |\mu - \mu' | \left(\sum_{j \in F} |z_i| \right) \left(\sum_{j' \in F'} |z_{j'}| \right).  
\end{align*} 
To calculate the function $f(x,y)$, we analyze the sum $\sum_{1 \le i < j < k \le m} |\det(a_i,a_j,a_k)|$, with $a_i = (x_i,y_i,z_i)$. 
Under our choice of $x$ and $y$, every vector $a_i$ has one of the four forms $z_i ( \lambda, \mu, 1)$, $z_i (\lambda,\mu',1)$, $z_i (\lambda', \mu,1)$, or $z_i (\lambda',\mu',1)$. Hence, $|\det(a_i,a_j,a_k)|$ is $|z_i z_j z_k|$ times the determinant of a $3 \times 3$ matrix whose columns belong to  $\{\lambda,\lambda' \} \times \{\mu,\mu'\} \times \{1\}$. We have the partition
$\{1,\ldots,m\}=(E \cap F)\cup(E \cap F')\cup(E' \cap F)\cup(E'  \cap F')$. If any two of the three indices $\{i, j,k\}$ belong to the same part of the partition then two of the three vectors $a_i,a_j,a_k$ coincide and, hence, $|\det(a_i,a_j,a_k)| =0$. If the indices $i,j,k$ belong to three distinct parts, then $|\det(a_i,a_j,a_k)| = |z_i z_j z_k| \cdot |\lambda - \lambda'| \cdot | \mu - \mu'|$. 
This follows from the observation that 
$\frac{1}{2} | \det(a_i/z_i, a_j/z_j,a_k/z_k)| = \frac{1}{2} |\lambda - \lambda' | \cdot |\mu - \mu'|$ is the area of a triangle 
formed by three vertices of a rectangle with vertex set $\{\lambda,\lambda'\} \times \{\mu,\mu'\}$ and edge lengths $|\lambda- \lambda'|$ and $|\mu - \mu'|$.

Now, denoting 
\begin{align*} 
		s_1  & := \sum_{i \in E \cap F} |z_i|, &  s_2 & := \sum_{i \in E \cap F' } |z_i|, 
		\\ s_3 & := \sum_{i \in E'  \cap F} |z_i|, & s_4 & := \sum_{i \in E'  \cap F' } |z_i|
\end{align*} 
we can express the two sides of the inequality $f(x,y) \le g(x) g(y)$ as follows: 
\begin{align*} 
	f(x,y) & = | \lambda - \lambda'  | \cdot | \mu - \mu' | \cdot \left( s_2 s_3 s_4 + s_1 s_3 s_4 + s_1 s_2 s_4 + s_1 s_2 s_3 \right) (s_1 + s_2 + s_3 + s_4), \\
	g(x) g(y) & = | \lambda - \lambda'  | \cdot | \mu - \mu' | \cdot (s_1 + s_2) (s_3 + s_4) (s_1 + s_3) (s_2 + s_4). 
\end{align*} 
Finally, $f(x,y) \le g(x) g(y)$ follows from the non-negativity of 
\begin{align*} 
	& (s_1 + s_2) (s_3 + s_4) (s_1 + s_3) (s_2 + s_4) - \left( s_2 s_3 s_4 + s_1 s_3 s_4 + s_1 s_2 s_4 + s_1 s_2 s_3 \right) (s_1 + s_2 + s_3 + s_4) 	
\\  = & (s_1 s_4 - s_2 s_3)^2. 
\end{align*} 
 \end{proof}

\begin{Rem} 
	The fact that \eqref{the:ineq} ultimately follows from the non-negativity of the square $(s_1 s_4 - s_2 s_3)^2$ might come as a surprise. In truth, this came as a surprise to the authors themselves. We simply followed the path suggested by Proposition~\ref{P:abstract}, until we eventually saw that \eqref{the:ineq} reduces to  $(s_1 s_4 - s_2 s_3)^2\geq 0$.
\end{Rem} 

\begin{Rem}
	In hindsight, we see that the validity of \eqref{the:ineq} for zonotopes generated by at most $m$ segments follows from the validity of \eqref{the:ineq} for zonotopes generated by at most $4$ segments, because the reduction we carry out in the proof of Lemma~\ref{L:FMMZ} corresponds to checking \eqref{the:ineq} for $A$ being the zonotope  $A = s_1 [0, (\lambda,\mu,1) ] + s_2 [0, (\lambda,\mu',1) ] + s_3 [0, (\lambda',\mu,1) ] + s_4 [0, (\lambda',\mu',1)] $ and $B$ and $C$ being the segments $B = [0,e_1]$ and $C = [0,e_2]$, respectively. Our proof of Lemma~\ref{L:FMMZ} also implies that the constant $3/2$ in \eqref{the:ineq} is optimal. To see this, note that the equality for $A$, $B$, and $C$ as above is attained when $s_1 s_4  =  s_2 s_3$, for example, for $s_1 = \cdots =s_4 = 1$. 
\end{Rem} 

\begin{Rem}
	Lemma~\ref{L:FMMZ} has a  natural algebraic-geometric interpretation in terms of the Grassmannian. The system of $d \times d$ minors of a given $d \times n$ matrix $A = (a_1,\ldots,a_n) \in \R^{n \times d}$ is a point of the so-called $(d,n)$-Grassmannian. More formally, let $d \le n$ and let $\binom{[n]}{d}$ be the family of $d$ element subsets of $[n] := \{1,\ldots,n\}$. If $I = \{i_1,\ldots,i_d\} \in \binom{[n]}{d}$ with $1 \le i_1 < \cdots < i_d \le n$, then we use $A_I$ to denote the submatrix $(a_{i_1},\ldots, a_{i_d}) \in \R^{d \times d}$ of $A$. The real $(d,n)$-Grassmannian can be defined as $\Gr(d,n) := \left\{ \Bigl(\det(A_I)\Bigr)_{I \in \binom{[n]}{d}} \colon A \in \R^{d \times n}\right\} \subseteq \R^{\binom{n}{d}}$.  It is known that $\Gr(d,n)$ is a homogeneous algebraic variety defined by a system of quadratic equations, the so-called \emph{Grassmann-Plücker relations}, see for example \cite[Section 2.4]{OrientMat}. In the context of studying inequalities for mixed volumes of zonotopes, the image of $\Gr(d,n)$ under the 
	component-wise absolute value map $\Abs : \R^{ \binom{n}{d} } \to \R^{\binom{n}{d}}$, which sends $g = \bigl(g_I \bigr)_{I \in \binom{[n]}{d} }$ to $\Abs(g)= \bigl( | g_I| \bigr)_{I \in \binom{[n]}{d}}$,  plays a key role. By considering the $3 \times 3$ minors of the matrix $A = (a_1,\ldots,a_m, e_1, e_2) \in \R^{3 \times (m+2)}$, one can see that the assertion of Lemma~\ref{L:FMMZ} is equivalent to the validity of the quadratic inequality 
	\[
		\biggl( \sum_{ I \in \binom{[m]}{3} } q_I \biggr) \biggl( \sum_{i=1}^m q_{ \{i,m+1,m+2\} } \biggr) \le \biggl( \sum_{ S \in \binom{[m]}{2} } q_{S \cup \{m+1\}} \biggr) \biggl( \sum_{ T \in \binom{[m]}{2} } q_{T \cup \{m+2\}} \biggr) 
	\] 
	for all $q \in \Abs(\Gr(3,m+2))$. As observed in the previous remark, the case $m=4$ implies the case of an arbitrary $m \ge 4$. In this case, we have the inequality 
	\begin{align*} 
		& ( q_{234} + q_{134} + q_{124} + q_{123} ) (q_{156} + q_{256} + q_{356} + q_{456}) 
		\\ \le & (q_{125} + q_{135} + q_{145} + q_{235} + q_{245} + q_{345} )  (q_{126} + q_{136} + q_{146} + q_{236} + q_{246} + q_{346} ) 
	\end{align*} 
	valid for $q \in \Abs(\Gr(3,6))$. 
\end{Rem}



\bibliographystyle{acm} 
\bibliography{lit}

\end{document}